\documentclass[11pt]{article}

\usepackage[T1]{fontenc}
\usepackage[utf8]{inputenc}
\usepackage{lmodern}
\usepackage{amsmath,amssymb,amsthm,mathtools}
\usepackage{enumitem}
\usepackage[colorlinks=true,linkcolor=blue,citecolor=blue,urlcolor=blue]{hyperref}
\usepackage[margin=1.12in]{geometry}
\usepackage{microtype}

\newtheorem{theorem}{Theorem}[section]
\newtheorem{lemma}[theorem]{Lemma}

\theoremstyle{definition}

\theoremstyle{remark}

\newcommand{\E}{\mathbb E}
\newcommand{\Pp}{\mathbb P}
\newcommand{\Z}{\mathbb Z}

\newcommand{\range}{\operatorname{range}}

\title{North-East Lattice Paths with Few Collinear Vertices}
\author{Samuel Korsky}
\date{\today}

\begin{document}
\maketitle

\begin{abstract}
\noindent
Let $A(k)$ be the largest possible number of moves in a north-east lattice path whose visited vertices contain no $k$ collinear points.  Gerver (1979) and Gerver and Ramsey (1979) gave lower and upper bounds on $A(k)$ of the form
\[
\exp\left(\Omega(\log(k)^2)\right)\le A(k)\le
\exp\left(O(k^4)\right).
\]
Improving upon these results, we show that
\[
\exp\left(\Omega(k^{1/3})\right)\le A(k)\le
\exp\left(\left(\frac{2}{e}+o(1)\right)(k-1)^2\right).
\]
\end{abstract}

\section{Introduction}

A north-east lattice path is a finite path in $\Z^2$ starting at the origin whose steps are either $(1,0)$ or $(0,1)$. For $k\ge2$, let $A(k)$ denote the maximum numbers of moves in such paths whose visited vertices contain no $k$ collinear points. This paper investigates the asymptotic growth of $A(k)$.

\bigskip
\noindent
The question is related to, but distinct from, the classical no-three-in-line problem. In the classical problem one seeks large subsets of a finite grid with no three collinear points; see the account in Brass, Moser, and Pach \cite{BMP}. Here the points must occur in the order forced by a monotone path, which makes the problem closer to Ramsey-type questions about lattice walks.

\bigskip
\noindent
Brown asked whether every sufficiently long north-east path must contain $k$ collinear points for each fixed $k$, and Montgomery gave a positive non-quantitative answer \cite{Montgomery}. Gerver and Ramsey later proved an effective theorem for finite-step lattice walks in the plane \cite{GerverRamsey}. In the north-east case their argument gives a finite upper bound with roughly fourth-power exponential dependence on $k$. Gerver also constructed long examples, showing in particular that $A(k)$ grows faster than any fixed power of $k$; his construction gives $A(k)\ge \exp(\Omega((\log k)^2))$ \cite{Gerver}.

\bigskip
\noindent
Recent computational work of Barnoff and Bright studies the small values using satisfiability methods \cite{BarnoffBright}. Related collinearity questions for lattice-point sequences were considered by Pomerance \cite{Pomerance}, and a higher-dimensional variant was recently revisited by Lidbetter \cite{Lidbetter}.

\bigskip
\noindent
Our first main result is the following upper bound.

\begin{theorem}[Upper bound]\label{thm:upper_intro}
As $k\to\infty$,
\[
\log A(k)\le \left(\frac{2}{e}+o(1)\right)(k-1)^2.
\]
\end{theorem}

\smallskip
\noindent
Our second main result is a lower bound.

\begin{theorem}[Weighted dyadic slope-field lower bound]\label{thm:lower_intro}
There is an absolute constant $c>0$ such that, for all sufficiently large $k$,
\[
\log A(k)\ge c k^{1/3}.
\]
\end{theorem}

\section{Proof Overview}\label{sec:overview}

\subsection{Upper bound}

The proof of Theorem \ref{thm:upper_intro} encodes the path by a binary word $w_0,\ldots,w_{n-1}$, where $w_i=1$ for a north step and $w_i=0$ for an east step, and sets $S_t=\sum_{i<t}w_i$. Thus $S_t$ counts the number of north steps before time $t$. We work in the affine coordinates $(t,S_t)$ rather than the usual coordinates $(t-S_t,S_t)$; this preserves collinearity. For an interval $I=[u,u+L]$, define its density by
\[
\alpha(I)=\frac{S_{u+L}-S_u}{L},
\]
the fraction of steps in $I$ which are north, or equivalently the slope of the chord joining the endpoints of the path segment in the $(t,S_t)$-plane.

\bigskip
\noindent
On such an interval, choose a rational approximation $p/q$ to $\alpha=\alpha(I)$ and define
\[
D_t=q(S_{u+t}-S_u)-pt,\qquad 0\le t\le L.
\]
The levels $D_t=c$ are intersections with translates of the rational-slope line of slope $p/q$. Since no line contains $k$ visited vertices, each value of $D_t$ occurs at most $k-1$ times, so the integer sequence $D_0,\ldots,D_L$ has large range. But
\[
D_L=qL\left(\alpha-\frac pq\right),
\]
so a good approximation $p/q\approx\alpha$ makes the endpoint value small. Subtracting the chord from $(0,D_0)$ to $(L,D_L)$ therefore forces some interior value of
\[
F_t=D_t-\frac{t}{L}\cdot D_L
\]
to be large. Since $F_t=q(S_{u+t}-S_u-\alpha t)$, this means that either the prefix $[u,u+t]$ or the suffix $[u+t,u+L]$ has density farther from $1/2$ than $I$. This is the density-increment step.

\bigskip
\noindent
The constant comes from choosing $p/q$ uniformly well for every possible density $\alpha$. A basic Dirichlet choice would give $\lambda\eta\ge (1/8-o(1))(k-1)^{-2}$, where $\lambda$ is the relative length of the chosen subinterval and $\eta$ is its density change. Using the two Farey neighbors of order $h=k-1$ together with their mediant improves this to
\[
\lambda\eta\ge \left(\frac14-o(1)\right)(k-1)^{-2}.
\]
Iterating these density increments gives
\[
\log A(k)\le \left(\frac2e+o(1)\right)(k-1)^2.
\]

\subsection{Lower bound}

For the proof of Theorem \ref{thm:lower_intro} we work in the affine coordinates $(t,D)$, where
\[
D_t=2S_t-t.
\]
Then a north step changes $D_t$ by $+1$, an east step changes $D_t$ by $-1$, and the parity condition $D_t\equiv t\pmod 2$ is exactly what is needed to recover the original north-east path from $S_t=(D_t+t)/2$. We split time into $N=2^J$ coarse intervals of a fixed constant length $T_0$. The construction first chooses random boundary heights
\[
H_i=D_{iT_0},\qquad 0\le i\le N,
\]
and then fills each coarse interval by a deterministic $\pm1$ bridge from $H_i$ to $H_{i+1}$.

\bigskip
\noindent
The boundary heights are chosen by a random dyadic superposition on the coarse index set $[0,N)$. For each dyadic interval
\[
I=[p2^s,(p+1)2^s)\subseteq[0,N),
\]
we choose an independent integer random variable $Z_I$ as follows:
\[
Z_I\sim \operatorname{Unif}\{-R_s,-R_s+1,\ldots,R_s\},
\qquad
R_s=\max\left(1,\left\lceil\frac{2^s}{J}\right\rceil\right).
\]
This variable contributes linearly across $I$; before rounding, it adds $Z_I/|I|$ to each coarse increment $H_{j+1}^*-H_j^*$ with $j\in I$, and contributes nothing outside $I$. Equivalently,
\[
H_i^*=\sum_I Z_I\cdot\frac{|I\cap[0,i)|}{|I|}.
\]
Since for each coarse cell $j$ there is only one dyadic interval of each scale containing $j$, the size of a single coarse increment is bounded by
\[
\sum_{s=0}^J \frac{R_s}{2^s}=O(1).
\]
After rounding the $H_i^*$ to the required parity, the resulting increments $H_{i+1}-H_i$ remain bounded by an absolute constant, so a fixed constant-length $\pm1$ bridge suffices inside every coarse interval.

\bigskip
\noindent
The line-count estimate comes from a diagonal anti-concentration argument. Suppose selected vertices lie in active coarse blocks
\[
i_1<\cdots<i_r,
\]
and write
\[
g_\alpha=i_{\alpha+1}-i_\alpha
\]
for the gaps between consecutive active blocks. For each gap, one can choose a dyadic interval $I_\alpha\subseteq[i_\alpha,i_{\alpha+1})$ of length comparable to $g_\alpha$. The associated random variable $Z_{I_\alpha}$ affects the raw height difference $H_{i_{\alpha+1}}^*-H_{i_\alpha}^*$ with coefficient exactly one, and affects the other selected raw gap differences with coefficient zero. Rounding does not preserve exact linearity, but it only changes each boundary height by at most one. Thus an exact equation for rounded height differences implies a bounded-width slab condition for the raw differences. Conditioning on all non-pivot variables then leaves each pivot variable restricted to $O(1)$ possible integer values. Since $Z_{I_\alpha}$ has support size comparable to $\max(1,g_\alpha/J)$, the probability of satisfying prescribed height differences across all gaps is at most
\[
\prod_{\alpha=1}^{r-1}
C_0\min\left(1,\frac{J}{g_\alpha}\right),
\]
where $C_0$ is an absolute constant.

\bigskip
\noindent
It remains to sum this probability bound over the possible active coarse blocks. There are at most $N$ choices for the first active block $i_1$. Once $i_1$ is fixed, the later active blocks are determined by the successive gap lengths $g_1,\ldots,g_{r-1}$. Thus the total contribution of the gap-probability factors is bounded by
\[
N\prod_{\alpha=1}^{r-1}
\left(
\sum_{g=1}^{N} C_0\min\left(1,\frac{J}{g}\right)
\right).
\]
Since $N=2^J$,
\[
\sum_{g=1}^{N}\min\left(1,\frac{J}{g}\right)
\le
J+J\sum_{g=J+1}^{2^J}\frac1g
=O(J^2).
\]
Therefore, for another absolute constant $C_1$,
\[
\sum_{0\le i_1<\cdots<i_r<N}
\prod_{\alpha=1}^{r-1}
C_0\min\left(1,\frac{J}{i_{\alpha+1}-i_\alpha}\right)
\le
N(C_1J^2)^{r-1}.
\]
The remaining choices introduce only constant factors. Indeed, each coarse block has fixed length $T_0=O(1)$ and the increments $H_{i+1}-H_i$ are bounded independently of $J$. These factors are absorbed into the same $J^2$-based bound. Hence, for every fixed nonvertical line direction $L(t,D)=at+bD$, if $N_L(q)$ denotes the number of vertices on the level $L(t,D)=q$, then
\[
\E\left[\sum_q\binom{N_L(q)}{m}\right]
\le
T(C_2J^2)^m,\qquad m\le J,
\]
where $C_2$ is an absolute constant and $T=NT_0\asymp 2^J$ is the total path length.

\bigskip
\noindent
This moment estimate controls the largest number of vertices on any line of the fixed direction. Indeed, if some level $L(t,D)=q$ contained $k$ vertices, then the left-hand side would be at least $\binom{k}{m}$. Taking $m\asymp J$, applying Markov's inequality, and then union-bounding over all possible nonvertical line directions gives a realization in which every line contains at most $O(J^3)$ vertices. Since the path has length $T\asymp 2^J$, choosing $J\asymp k^{1/3}$ gives
\[
A(k)\ge \exp(c k^{1/3}).
\]

\section{Encoding and Elementary Facts}\label{sec:encoding}

Let $P$ be a north-east path with $n$ moves. Write its step word as $w_0,w_1,\ldots,w_{n-1}\in\{0,1\}$, where $w_i=1$ means north and $w_i=0$ means east. Define
\[
S_t=\sum_{i=0}^{t-1}w_i,\qquad 0\le t\le n.
\]
The usual vertex after $t$ moves is $(t-S_t,S_t)$. We instead use the affine coordinates
\[
P_t=(t,S_t).
\]
The map $(t,S)\mapsto(t-S,S)$ is invertible and affine, so it preserves collinearity.

\bigskip
\noindent
For a rational slope $p/q$ in the $(t,S)$-plane, with $q\ge1$ and $\gcd(p,q)=1$, define
\[
D_t^{p,q}=qS_t-pt.
\]
The level sets $D_t^{p,q}=c$ are the intersections of the path with lines of slope $p/q$ in the affine coordinates. Thus, if a path contains no $k$ collinear visited vertices and $h=k-1$, every level of every sequence $D_t^{p,q}$ is hit at most $h$ times. Horizontal lines in the original coordinates are the levels of $S_t$, and vertical lines in the original coordinates are the levels of $t-S_t$.

\begin{lemma}[Range from bounded level multiplicity]\label{lem:range}
Let $x_0,x_1,\ldots,x_L$ be integers, and suppose no integer value occurs more than $h$ times. Then
\[
\range(x_t:0\le t\le L)\ge \frac{L+1}{h}-1.
\]
\end{lemma}

\begin{proof}
The sequence assumes at least $(L+1)/h$ distinct integer values. A set of $m$ distinct integers has range at least $m-1$.
\end{proof}

\section{The Upper Bound}\label{sec:upper}

Throughout this section set $h=k-1$. If $I=[u,u+L]$ is an interval of times, define its density by
\[
\alpha(I)=\frac{S_{u+L}-S_u}{L}.
\]
The central density range is
\[
\mathcal C_h=\left[\frac1{2h},1-\frac1{2h}\right].
\]

\begin{lemma}[Farey-neighbor-plus-mediant approximation]\label{lem:farey_mediant}
For every $\alpha\in[0,1]$ and every integer $h\ge2$, there are coprime integers $p,q$ with $0\le p\le q$ and $1\le q\le2h$ such that
\[
\frac1{2q}\left(\frac1h-q\left|\alpha-\frac pq\right|\right)
\ge \frac1{4h^2}.
\]
\end{lemma}

\begin{proof}
If $\alpha$ is a Farey fraction of order $h$, take $p/q=\alpha$ in lowest terms. Then $q\le h$ and the left-hand side is at least $1/(2h^2)$.

\bigskip
\noindent
Otherwise let
\[
\frac aq<\alpha<\frac{a'}{q'}
\]
be the consecutive Farey fractions of order $h$ which bracket $\alpha$. Then $q,q'\le h$, $q+q'>h$, and $a'q-aq'=1$. Write
\[
\alpha=(1-\theta)\cdot\frac aq+\theta\cdot\frac{a'}{q'},
\qquad 0\le\theta\le1.
\]
Put
\[
x=\frac qh,
\qquad
 y=\frac{q'}h,
\qquad s=x+y.
\]
Then $0<x,y\le1$ and $s>1$. The two endpoint fractions give, after multiplying the desired expression by $h^2$,
\[
C_1=\frac{1-\theta/y}{2x},
\qquad
C_2=\frac{1-(1-\theta)/x}{2y}.
\]
Now also consider the mediant
\[
\frac{a+a'}{q+q'}.
\]

\smallskip
\noindent
It is reduced and has denominator $q+q'\le2h$. Since its relative position between the two Farey neighbors is $y/s$, the corresponding scaled quantity is
\[
C_3=\frac1{2s}-\frac{|\theta-y/s|}{2xy}.
\]
We claim that
\[
\max(C_1,C_2,C_3)\ge\frac14.
\]
Indeed, if $C_1<1/4$ and $C_2<1/4$, then
\[
\theta>y\left(1-\frac x2\right)
\]
and
\[
\theta<1-x+\frac{xy}{2}.
\]
On the other hand, $C_3\ge1/4$ is equivalent to
\[
\left|\theta-\frac ys\right|\le \frac{xy(2-s)}{2s}.
\]
The lower endpoint of this interval satisfies
\[
\frac ys-\frac{xy(2-s)}{2s}-y\left(1-\frac x2\right)
=\frac{y(x-1)(s-1)}{s}\le0,
\]
and the upper endpoint satisfies
\[
\frac ys+\frac{xy(2-s)}{2s}-\left(1-x+\frac{xy}{2}\right)
=-\frac{x(y-1)(s-1)}{s}\ge0.
\]
Thus any $\theta$ for which both endpoint choices give less than $1/4$ is covered by the mediant choice. This proves the claim and hence the lemma.
\end{proof}

\begin{lemma}[Density increment]\label{lem:density_increment}
There are absolute constants $C_0,C_1>0$ with the following property. Let a north-east path contain no $k$ collinear vertices. Let $I=[u,u+L]$ be an interval of length $L\ge C_0h^2$ with density $\alpha\in\mathcal C_h$. Then there is a proper subinterval $I'\subset I$ whose density lies farther from $1/2$ than $\alpha$ does. If $\lambda=|I'|/|I|$ and $\eta=|\alpha(I')-\alpha|$, then
\[
\lambda\eta\ge \frac1{4h^2}-\frac{C_1}{L}.
\]
\end{lemma}

\begin{proof}
Apply Lemma \ref{lem:farey_mediant} to choose $p/q$. For $0\le t\le L$, define
\[
D_t=q(S_{u+t}-S_u)-pt.
\]
Every level of $D_t$ is hit at most $h$, since a level of $D_t$ is a line of slope $p/q$ in the affine coordinates $(t,S_t)$. Lemma \ref{lem:range} gives
\[
\range(D_t:0\le t\le L)\ge \frac{L+1}{h}-1=\frac Lh-O(1).
\]
Moreover,
\[
|D_L|=qL\left|\alpha-\frac pq\right|.
\]
Set
\[
F_t=D_t-\frac{t}{L}D_L.
\]
Then $F_0=F_L=0$, and subtracting the chord can shrink the range by at most $|D_L|$. Hence
\[
\range(F_t:0\le t\le L)
\ge L\left(\frac1h-q\left|\alpha-\frac pq\right|\right)-O(1).
\]
By Lemma \ref{lem:farey_mediant}, the main term on the right is at least $Lq/(2h^2)$, and hence at least $L/(2h^2)$, before the absolute $O(1)$ loss.  Taking $C_0$ sufficiently large ensures that the range lower bound is positive.  Since $F_0=F_L=0$, there is therefore a time $0<t<L$ with
\[
|F_t|\ge \frac L2\left(\frac1h-q\left|\alpha-\frac pq\right|\right)-O(1).
\]
A direct calculation gives
\[
F_t=q\bigl(S_{u+t}-S_u-\alpha t\bigr).
\]
Let
\[
\alpha_-:=\frac{S_{u+t}-S_u}{t}=\alpha+\frac{F_t}{qt}
\]
and
\[
\alpha_+:=\frac{S_{u+L}-S_{u+t}}{L-t}=\alpha-\frac{F_t}{q(L-t)}.
\]
If $\alpha\ge1/2$, choose the prefix when $F_t>0$ and the suffix when $F_t<0$; if $\alpha\le1/2$, choose the prefix when $F_t<0$ and the suffix when $F_t>0$. In all cases the chosen subinterval has density farther from $1/2$ than $\alpha$, and if its relative length is $\lambda$ and its density increment is $\eta$, then
\[
\lambda\eta=\frac{|F_t|}{qL}
\ge \frac1{2q}\left(\frac1h-q\left|\alpha-\frac pq\right|\right)-\frac{C_1}{L}
\]
after increasing the absolute constant $C_1$ if necessary.  Lemma \ref{lem:farey_mediant} gives the claimed lower bound.  The choice of $C_0$ ensures that the right-hand side is positive, so the time $t$ used above is genuinely interior and the chosen interval is proper.
\end{proof}

\begin{lemma}[Stopping outside the central range]\label{lem:stop}
Let $I=[u,u+L]$ be an interval in a path with no $k$ collinear vertices. If $L>2h$ and $\alpha(I)<1/(2h)$, then the path has more than $h$ vertices on a horizontal line. If $L>2h$ and $\alpha(I)>1-1/(2h)$, then the path has more than $h$ vertices on a vertical line in the original $(x,y)$-coordinates.
\end{lemma}

\begin{proof}
If $\alpha(I)<1/(2h)$, then the number of north steps in $I$ is less than $L/(2h)$, so the vertices of $I$ lie on fewer than $L/(2h)+1$ horizontal levels $S_t=c$. If every such level contained at most $h$ vertices, the number of vertices in the interval would be less than $L/2+h$. For $L>2h$ this is less than $L+1$, a contradiction. Thus some horizontal level contains more than $h$ vertices. The case $\alpha(I)>1-1/(2h)$ is the same after replacing north steps by east steps; equivalently, use levels of $t-S_t$, which are vertical lines in the original coordinates.
\end{proof}

\begin{proof}[Proof of Theorem \ref{thm:upper_intro}]
Let $P$ be any north-east path with $n$ moves and no $k$ collinear visited vertices. We prove the asserted bound for $n$, uniformly over such paths.

\bigskip
\noindent
Construct nested intervals
\[
I_0\supset I_1\supset I_2\supset\cdots.
\]
Start with $I_0=[0,n]$. Put $R_h=h^4$. If $|I_i|\le R_h$, stop. If the density of $I_i$ is outside the central range, Lemma \ref{lem:stop} implies $|I_i|\le2h$, so in particular $|I_i|\le R_h$ for large $h$, and we stop. Otherwise $|I_i|>R_h$ and the density lies in the central range, so Lemma \ref{lem:density_increment} applies; let $I_{i+1}$ be the subinterval it produces.

\bigskip
\noindent
Write $L_i=|I_i|$, $\lambda_i=L_{i+1}/L_i$, and $\eta_i=|\alpha(I_{i+1})-\alpha(I_i)|$. By construction the densities move monotonically away from $1/2$, hence
\[
\sum_i\eta_i\le \frac12.
\]
For all nonterminal steps $L_i>R_h=h^4$.  For all sufficiently large $h$ this implies $L_i\ge C_0h^2$, so Lemma \ref{lem:density_increment} applies and gives
\[
\lambda_i\eta_i\ge \gamma_h,
\qquad
\gamma_h=\frac1{4h^2}-\frac{C_1}{h^4}
=\left(\frac14-o(1)\right)\frac1{h^2}.
\]
Therefore
\[
\log\frac{L_i}{L_{i+1}}=\log\frac1{\lambda_i}\le \log\frac{\eta_i}{\gamma_h}.
\]
Since $\log x\le x/e$ for all $x>0$,
\[
\sum_i\log\frac{L_i}{L_{i+1}}
\le \sum_i\frac{\eta_i}{e\gamma_h}
\le \frac{1}{2e\gamma_h}
=\left(\frac{2}{e}+o(1)\right)h^2.
\]
At termination, $L_m\le R_h=h^4$, so $\log L_m=O(\log h)=o(h^2)$. Since
\[
\log n=\log L_0=\sum_{i=0}^{m-1}\log\frac{L_i}{L_{i+1}}+\log L_m,
\]
we get
\[
\log n\le \left(\frac{2}{e}+o(1)\right)h^2.
\]
As $h=k-1$, this proves the theorem.
\end{proof}

\section{The Lower Bound}\label{sec:dyadic_lower}

We now prove Theorem \ref{thm:lower_intro}.  The proof utilizes a weighted dyadic slope-field construction.

\subsection{Lower-bound coordinates and parameters}

For the lower bound it is more convenient to use the diagonal coordinate
\[
D_t=2S_t-t.
\]
Then $D_0=0$, every step changes $D$ by $\pm1$, and
\[
D_t\equiv t\pmod 2.
\]
Conversely, every such $\pm1$-walk gives a north-east path by
\[
S_t=\frac{D_t+t}{2}.
\]
The affine map $(t,S)\mapsto(t,2S-t)$ preserves collinearity.  For line-counting it is enough to consider affine lines containing at least two
lattice vertices of the path.  Any such nonvertical line has rational slope, and therefore has an equation
\[
L(t,D)=at+bD=q,
\qquad b\ne0,
\]
with $a,b,q\in\Z$ after multiplying by a common denominator.  Lines $t=\text{constant}$ contain at most one path vertex and are harmless.

\bigskip
\noindent
Let $J$ be a large integer and put
\[
N=2^J.
\]
Fix a sufficiently large even absolute integer $T_0$; for definiteness one
may take $T_0=100$.  Put
\[
T=NT_0.
\]
The path will consist of $N$ coarse blocks, each of length $T_0$.

\subsection{The weighted random boundary process}

Let $\mathcal D$ be the family of dyadic intervals of coarse indices inside
$[0,N)$:
\[
I=[p2^s,(p+1)2^s),
\qquad 0\le s\le J,
\]
with $I\subseteq[0,N)$.  Write $|I|=2^s$, and define the scale amplitude
\[
R_s=\max\left(1,\left\lceil \frac{2^s}{J}\right\rceil\right).
\]
For every $I\in\mathcal D$ of length $2^s$, let
\[
Z_I\sim \operatorname{Unif}\{-R_s,-R_s+1,\ldots,R_s\}
\]
independently.  Define raw boundary heights by
\[
H_i^*=\sum_{I\in\mathcal D}Z_I\cdot\frac{|I\cap[0,i)|}{|I|},
\qquad 0\le i\le N.
\]
Thus $H_0^*=0$, and the raw increment over coarse cell $i$ is
\[
H_{i+1}^*-H_i^*=\sum_{I\ni i}\frac{Z_I}{|I|}.
\]
For each fixed $i$, exactly one dyadic interval of each scale contains
$i$.  Hence
\[
|H_{i+1}^*-H_i^*|
\le \sum_{s=0}^J \frac{R_s}{2^s}.
\]
The terms with $2^s\le J$ contribute at most $\sum_s 2^{-s}\le2$, while for $2^s>J$ we have $R_s\le 2\cdot 2^s/J$ and hence each such scale contributes at most $2/J$.  There are at most $J+1$ scales, so
\[
|H_{i+1}^*-H_i^*|\le B_0
\]
for an absolute constant $B_0$; for example, the preceding estimates allow one to take $B_0=6$ for all $J\ge1$.

\bigskip
\noindent
Now round to even integers:
\[
H_i=2\left\lfloor \frac{H_i^*}{2}+\frac12\right\rfloor.
\]
Then
\[
|H_i-H_i^*|\le1,
\]
and therefore
\[
\delta_i:=H_{i+1}-H_i
\]
is even and satisfies
\[
|\delta_i|\le B_0+2.
\]
We choose the fixed even constant $T_0$ large enough that
\[
B_0+2\le \frac{T_0}{10}.
\]
For every even $\delta$ with $|\delta|\le T_0/10$, fix one deterministic $\pm1$-walk
\[
\phi_\delta(0),\phi_\delta(1),\ldots,\phi_\delta(T_0)
\]
such that
\[
\phi_\delta(0)=0,
\qquad
\phi_\delta(T_0)=\delta.
\]
This exists because $T_0$ and $\delta$ are even and $|\delta|\le T_0$. Define the global walk by
\[
D_{iT_0+u}=H_i+\phi_{\delta_i}(u),
\qquad 0\le i<N,
\quad 0\le u\le T_0.
\]
The definitions agree at block boundaries, since
\[
H_i+\phi_{\delta_i}(T_0)=H_i+\delta_i=H_{i+1}.
\]
Thus $D_t$, $0\le t\le T$, is a global $\pm1$-walk.  Also $H_i$ and $T_0$ are even, and every $\pm1$-walk from $0$ satisfies $\phi_\delta(u)\equiv u\pmod2$, so
\[
D_{iT_0+u}\equiv u\equiv iT_0+u\pmod2.
\]
Consequently $S_t=(D_t+t)/2$ is an integer north-east path.

\bigskip
\noindent
In the moment calculation we use the half-open vertex set
\[
0\le t<T.
\]
Thus every time used in the moment calculation has a unique representation
\[
t=iT_0+u,
\qquad 0\le i<N,
\qquad 0\le u<T_0.
\]
The bridge definition above also includes the endpoint $u=T_0$, but those endpoints are not double-counted in the half-open moment calculation.  The terminal endpoint $t=T$ will be added only at the end; it can increase any line count by at most one.

\subsection{A weighted diagonal dyadic rank lemma}

We first record the same dyadic covering fact used before.

\begin{lemma}[Dyadic subinterval inside an interval]\label{lem:dyadic_subinterval}
Every integer interval $G=[u,v)$ of length $g=v-u\ge1$ contains a dyadic interval $I\subseteq G$ with
\[
|I|\ge g/4.
\]
\end{lemma}

\begin{proof}
If $g=1$, take $I=G$.  If $g\ge2$, choose a power of two $\ell$ with $\ell\le g/2<2\ell$.  Then $\ell>g/4$.  Let
\[
s=\ell\left\lceil\frac{u}{\ell}\right\rceil.
\]
Then $s$ is a multiple of $\ell$, $u\le s<u+\ell$, and $s+\ell\le u+2\ell\le v$.  Hence the dyadic interval $[s,s+\ell)$ lies inside $[u,v)$ and has length $\ell\ge g/4$.
\end{proof}

\smallskip
\noindent
The next lemma is the only place where the rounding step enters the probabilistic estimate.  The raw heights $H_i^*$ are linear functions of independent dyadic variables.  The rounded heights $H_i$ are no longer linear, but the deterministic bound $|H_i-H_i^*|\le1$ means that every exact rounded gap equation implies a raw gap equation with error at most $2$.  The pivot argument is therefore applied to bounded-width slabs in the raw variables, where the independence of the original variables $Z_I$ is still available.

\begin{lemma}[Weighted gap rank estimate]\label{lem:gap_rank}
Let
\[
0\le i_1<i_2<\cdots<i_r<N
\]
be distinct coarse indices, and put
\[
g_\alpha=i_{\alpha+1}-i_\alpha,
\qquad 1\le\alpha<r.
\]
For arbitrary real numbers $c_1,\ldots,c_{r-1}$,
\[
\Pp\left(
H_{i_{\alpha+1}}-H_{i_\alpha}=c_\alpha
\text{ for all }1\le\alpha<r
\right)
\le
\prod_{\alpha=1}^{r-1} C\min\left(1,\frac{J}{g_\alpha}\right),
\]
where $C$ is an absolute constant.
\end{lemma}

\begin{proof}
Let
\[
G_\alpha=[i_\alpha,i_{\alpha+1})
\]
be the $\alpha$-th gap.  By Lemma \ref{lem:dyadic_subinterval}, choose a dyadic interval
\[
I_\alpha\subseteq G_\alpha
\]
with $|I_\alpha|\ge g_\alpha/4$.  The gaps $G_\alpha$ are disjoint, so the chosen intervals $I_\alpha$ are distinct; hence the variables $Z_{I_1},\ldots,Z_{I_{r-1}}$ are independent.

\bigskip
\noindent
For any integer interval $G=[u,v)$,
\[
H_v^*-H_u^*
=\sum_{I\in\mathcal D}Z_I\cdot\frac{|I\cap G|}{|I|}.
\]
Since $I_\alpha\subseteq G_\alpha$, the coefficient of $Z_{I_\alpha}$ in $H_{i_{\alpha+1}}^*-H_{i_\alpha}^*$ is exactly $1$.  If $\beta\ne\alpha$, then $I_\alpha\cap G_\beta=\varnothing$, so the coefficient of $Z_{I_\alpha}$ in the $\beta$-th gap difference is $0$.  Thus the selected pivot variables enter the system diagonally.

\bigskip
\noindent
If
\[
H_{i_{\alpha+1}}-H_{i_\alpha}=c_\alpha,
\]
then, since $|H_i-H_i^*|\le1$,
\[
H_{i_{\alpha+1}}^*-H_{i_\alpha}^*\in[c_\alpha-2,c_\alpha+2].
\]
Condition on every random variable except the pivots $Z_{I_\alpha}$.  For each $\alpha$, the corresponding slab restricts $Z_{I_\alpha}$ to an interval of length at most $5$, and hence to at most $6$ integer values. If $|I_\alpha|=2^s$, then $Z_{I_\alpha}$ is uniform on an integer interval of size $2R_s+1$.  Since $|I_\alpha|\ge g_\alpha/4$,
\[
\frac{1}{2R_s+1}\le C\min\left(1,\frac{J}{g_\alpha}\right).
\]
The conditional probability of the $\alpha$-th slab is therefore at most $C\min(1,J/g_\alpha)$.  Multiplying over the independent pivots gives the claim.
\end{proof}

\subsection{A binomial moment estimate}

Fix an integer affine functional
\[
L(t,D)=at+bD,
\qquad b\ne0.
\]
For $q\in\Z$, define
\[
N_L(q)=\#\{0\le t<T:L(t,D_t)=q\}.
\]

\begin{lemma}[Fixed-direction binomial moment]\label{lem:dyadic_moment}
There is an absolute constant $C$ such that, for every $1\le m\le J$,
\[
\E\left[\sum_q\binom{N_L(q)}{m}\right]
\le
T(CJ^2)^m.
\]
The estimate is uniform in $a,b\in\Z$ with $b\ne0$.
\end{lemma}

\begin{proof}
The random variable
\[
\sum_q\binom{N_L(q)}{m}
\]
counts unordered $m$-subsets of the half-open vertex set whose $L$-values are all equal.  For such an $m$-set $\mathcal A$, write $E_\mathcal A$ for this equal-level event.  We sum $\Pp(E_\mathcal A)$ over all $\mathcal A$.

\bigskip
\noindent
Suppose first that $\mathcal A$ lies in a single coarse block.  There are $N$ choices of the block and at most $T_0^m$ choices of local vertices.  As $T_0$ is an absolute constant, the total contribution of these sets is at most
\[
NT_0^m\le T C^m\le T(CJ^2)^m.
\]
Now suppose $\mathcal A$ occupies exactly $r\ge2$ distinct coarse blocks
\[
i_1<i_2<\cdots<i_r.
\]
For fixed active blocks, the number of ways to choose the local vertices, after summing over the positive occupancies of the $r$ blocks, is at most
\[
2^mT_0^m\le C^m.
\]
For each active block $i_\alpha$, choose the least selected local time in that block and call it $u_\alpha$.  This canonical choice introduces no additional counting factor.  Using one representative from each active block only gives a necessary condition for $E_\mathcal A$, which is sufficient for an upper bound.  The active block increment
\[
\delta_{i_\alpha}=H_{i_\alpha+1}-H_{i_\alpha}
\]
has only $O(1)$ possible values, because $|\delta_i|\le B_0+2$.  We union-bound over all labels
\[
d_1,\ldots,d_r
\]
for these active increments, losing a factor at most $C^r$.

\bigskip
\noindent
Fix the active blocks, the local selected vertices, and labels $d_1,\ldots,d_r$.  On the event that the actual active increments satisfy $\delta_{i_\alpha}=d_\alpha$ for all $\alpha$, the representative vertex in block $i_\alpha$ has
\[
t_\alpha=i_\alpha T_0+u_\alpha,
\qquad
D_{t_\alpha}=H_{i_\alpha}+\phi_{d_\alpha}(u_\alpha).
\]
If all vertices of $\mathcal A$ have the same $L$-value, then in particular consecutive representatives have the same $L$-value.  Hence, on the event
\[
E_\mathcal A\cap
\{\delta_{i_\alpha}=d_\alpha\text{ for all }1\le\alpha\le r\},
\]
one has, for every $1\le\alpha<r$,
\[
H_{i_{\alpha+1}}-H_{i_\alpha}=c_\alpha(d_1,\ldots,d_r),
\]
where $c_\alpha(d_1,\ldots,d_r)$ is a deterministic real number depending only on the fixed data.  Explicitly it is obtained by dividing by $b\ne0$ in the equation
\[
\begin{aligned}
0={}&a\bigl((i_{\alpha+1}-i_\alpha)T_0+u_{\alpha+1}-u_\alpha\bigr)\\
&+b\bigl(H_{i_{\alpha+1}}-H_{i_\alpha}
 +\phi_{d_{\alpha+1}}(u_{\alpha+1})-\phi_{d_\alpha}(u_\alpha)\bigr).
\end{aligned}
\]
The possible labels cover the sample space, so
\[
\begin{aligned}
E_\mathcal A
&\subseteq
\bigcup_{d_1,\ldots,d_r}
\left(
\{\delta_{i_\alpha}=d_\alpha\text{ for all }1\le\alpha\le r\}
\cap
\left\{
H_{i_{\alpha+1}}-H_{i_\alpha}=c_\alpha(d_1,\ldots,d_r)
\text{ for all }1\le\alpha<r
\right\}\right)\\
&\subseteq
\bigcup_{d_1,\ldots,d_r}
\left\{
H_{i_{\alpha+1}}-H_{i_\alpha}=c_\alpha(d_1,\ldots,d_r)
\text{ for all }1\le\alpha<r
\right\}.
\end{aligned}
\]
The second inclusion merely drops the active-increment constraints.  Hence Lemma \ref{lem:gap_rank} gives
\[
\Pp(E_\mathcal A)
\le
C^r
\prod_{\alpha=1}^{r-1}C\min\left(1,\frac{J}{i_{\alpha+1}-i_\alpha}\right).
\]
It remains to sum over the active block indices.  Put
\[
w(g)=\min\left(1,\frac{J}{g}\right).
\]
Since $N=2^J$,
\[
\sum_{g=1}^N w(g)
\le J+J\sum_{g=J+1}^N \frac{1}{g}
\le C J^2.
\]
Therefore
\[
\sum_{0\le i_1<\cdots<i_r<N}
\prod_{\alpha=1}^{r-1}w(i_{\alpha+1}-i_\alpha)
\le
N(CJ^2)^{r-1}.
\]
The total contribution from $m$-sets occupying exactly $r\ge2$ coarse blocks is at most
\[
C^m\cdot C^r\cdot N(CJ^2)^{r-1}.
\]
Since $T=NT_0\asymp N$ and $r\le m$, this is at most
\[
T(CJ^2)^m
\]
after increasing $C$.  For the fixed value of $m$, summing over the possible numbers of active blocks $1\le r\le m$ contributes only a factor at most $m$; because $m\le J$ and $m\ge1$, this factor is absorbed by increasing the absolute constant in $(CJ^2)^m$.  The lemma follows.
\end{proof}

\subsection{From moments to line multiplicity}

We now complete the proof of the lower bound.

\begin{proof}[Proof of Theorem \ref{thm:lower_intro}]
The constructed half-open path lies in the deterministic box
\[
0\le t<T,
\qquad |D_t|\le T,
\]
because it is a $\pm1$-walk started at $0$.  Any nonvertical line through two lattice points in this box has a primitive normal vector $(a,b)$ with
\[
|a|,|b|\le 2T,
\qquad b\ne0.
\]
Thus there are at most $CT^2$ relevant nonvertical directions.  For each fixed direction, Lemma \ref{lem:dyadic_moment} gives
\[
\E\left[\sum_q\binom{N_L(q)}{m}\right] \le T(CJ^2)^m
\]
for every $1\le m\le J$.

\bigskip
\noindent
If some level has $N_L(q)\ge K$, then
\[
\binom{N_L(q)}{m}\ge \left(\frac{K}{em}\right)^m.
\]
Hence Markov's inequality gives, for a fixed direction,
\[
\Pp(\exists q:N_L(q)\ge K)
\le
T(CJ^2)^m\left(\frac{em}{K}\right)^m.
\]
Union-bounding over the $O(T^2)$ directions yields
\[
\Pp(\exists\text{ nonvertical line with at least }K
\text{ half-open vertices})
\le
CT^3\left(\frac{CmJ^2}{K}\right)^m.
\]
Choose
\[
m=J.
\]
Since
\[
T=T_0 2^J=\exp(O(J)),
\]
we may choose an absolute constant $C_1$ so large that, with
\[
K=C_1J^3,
\]
the last probability is less than $1$ for all sufficiently large $J$. Therefore there exists a half-open path for which every nonvertical line
contains at most
\[
C_2J^3
\]
vertices.  Vertical lines $t=\text{constant}$ contain at most one vertex. Adding the terminal endpoint $t=T$ increases every line count by at most one.  Thus, for all sufficiently large $J$, there is a full north-east path of length
\[
T=T_0 2^J\asymp 2^J
\]
with at most $C_3J^3$ vertices on any affine line.

\bigskip
\noindent
Given $k$, choose
\[
J=\left\lfloor \left(\frac{k}{2C_3}\right)^{1/3}\right\rfloor.
\]
Then $C_3J^3<k$ for all sufficiently large $k$, so the path has no $k$ collinear vertices.  Its length satisfies
\[
T\ge \exp(cJ)\ge \exp(c'k^{1/3})
\]
for absolute constants $c,c'>0$.  This proves
\[
A(k)\ge \exp(c'k^{1/3})
\]
which completes the proof.
\end{proof}

\section{Limitations of the Present Methods}\label{sec:limitations}

The preceding arguments leave a large gap between the lower and upper bounds.  We briefly indicate where the exponents $k^{1/3}$ and $k^2$ enter, and why the present methods do not appear to improve them without a new idea.

\bigskip
\noindent
For the upper bound, the iteration uses only two inputs: a one-step density increment
\[
\lambda_i\eta_i\ge \gamma_h,
\qquad
\gamma_h=\left(\frac14-o(1)\right)h^{-2},
\]
and the monotonicity condition that the densities move away from $1/2$, which gives
\[
\sum_i \eta_i\le \frac12.
\]
Together with $\log x\le x/e$, these two facts force
\[
\log n\le \frac{1}{2e\gamma_h}+o(h^2)
=\left(\frac2e+o(1)\right)h^2.
\]
Thus this density-increment scheme naturally stops at order $h^2=(k-1)^2$.  

\bigskip
\noindent
For the lower bound, the exponent $1/3$ comes from the estimate that a random path of length $T\asymp2^J$ can be made to have at most $O(J^3)$ vertices on any line.  The factor $J^3$ has two sources.  First, summing the anti-concentration weights over possible gaps gives
\[
\sum_{g=1}^{2^J}\min\left(1,\frac{J}{g}\right)=O(J^2).
\]
Second, the binomial moment parameter must be taken of order $J$ in order to beat the $\exp(O(J))$ number of possible directions and levels after the final union bound.  The moment calculation therefore gives a threshold of the form
\[
K\asymp mJ^2\asymp J^3.
\]
Choosing $J\asymp k^{1/3}$ is exactly what converts this into a path with no $k$ collinear vertices.

\bigskip
\noindent
Within the dyadic construction, the $J^2$ gap-summation cost is tied to the amplitude choice $R_s\asymp 2^s/J$, which is what keeps coarse increments bounded while preserving enough randomness at large scales.  Changing constants in the same construction therefore seems unlikely to move the exponent.  Potential improvements would have to introduce additional usable randomness, sharpen the line-counting step beyond the fixed-direction moment plus union bound, or replace the dyadic slope-field construction by a different model.

\section*{Acknowledgements}
The author was assisted by GPT-5.5 Pro in the preparation of this paper. The main construction ideas, including the dyadic-interval random variables in the lower bound and the density-increment framework in the upper bound, were due to the author. AI tools were used to check computations, improve the upper-bound constant by suggesting the use of the mediant of the relevant Farey fractions, and assist in drafting and editing the manuscript. The author is responsible for all statements and proofs in the final version.

\end{document}